\documentclass[12pt]{amsart}
\usepackage[latin1]{inputenc}
\usepackage{color}
\usepackage{enumerate}
\usepackage{amssymb}
\usepackage[all]{xy}
\usepackage{hyperref}

\numberwithin{equation}{section}

\newtheorem*{THM}{Main Theorem}

\newtheorem*{quest}{Question}

\newtheorem{thm}[equation]{Theorem}

\newtheorem{lem}[equation]{Lemma}

\theoremstyle{definition}
\newtheorem{defn}[equation]{Definition}
\theoremstyle{remark}
\newtheorem{rem}[equation]{Remark}

\newcommand{\norm}[1]{\left\Vert#1\right\Vert}

\newcommand{\set}[1]{\left\{#1\right\}}
\newcommand{\Real}{\mathbb R}

\newcommand{\To}{\longrightarrow}

\usepackage{color}
\usepackage{enumerate}
\begin{document}

\title{Complete metric spaces with property $(Z)$ are length spaces}
\author[A.\ Avil\'{e}s]{Antonio Avil\'es}
\address{Universidad de Murcia, Departamento de Matem\'{a}ticas, Campus de Espinardo 30100 Murcia, Spain.} 
\email{avileslo@um.es}

\author[G. Mart\'inez-Cervantes]{Gonzalo Mart\'inez-Cervantes}
\address{Universidad de Murcia, Departamento de Matem\'{a}ticas, Campus de Espinardo 30100 Murcia, Spain.}
\email{gonzalo.martinez2@um.es}

\thanks{Authors supported by projects MTM2014-54182-P and MTM2017-86182-P (Government of Spain, AEI/FEDER, EU) and by  project 19275/PI/14 (Fundaci\'on S\'eneca)}



\begin{abstract}

We prove that every complete metric space with property (Z) is a length space. These answers questions posed by Garc\'{i}a-Lirola, Proch\'{a}zka and Rueda Zoca, and by Becerra Guerrero, L\'{o}pez-P\'{e}rez and Rueda Zoca, related to the structure of Lipschitz-free Banach spaces of metric spaces.

\end{abstract}

\maketitle

\section{Introduction}

Let $M$ be a metric space with a distinguished point $0 \in M$. The set
$$Lip_0(M)= \set{ f:M \rightarrow \Real \mbox{ Lipschitz} : f(0)=0}$$
is a Banach space when endowed with the norm $$\norm{f} = \sup_{x \neq y} \frac{|f(x)-f(x)|}{d(x,y)}.$$

It turns out that $Lip_0(M)$ is always a dual space and that the space generated by all functionals of the form $\delta_m : Lip_0(M) \rightarrow \Real$ with $\delta_m(f)=f(m)$ for every $f \in Lip_0(M)$ is a predual which is usually denoted by $\mathcal{F}(M)$ and it is called the Lipschitz-free space of $M$.

During the last decades the relations between metric properties of $M$ and geometrical properties of $\mathcal{F}(M)$ have been deeply studied. In this paper we are going to focus on the following two properties (we denote the maximum of two real numbers $a$ and $b$ by $a \vee b$ and its minimum by $a \wedge b$):

\begin{defn}
	A pair $(x,y)$ of points of $M$ with $x\neq y$ is said to have property $(Z)$ if for every $\varepsilon>0$ there exists $z\in M\setminus\{x,y\}$  such that
	$$d(x,z) + d(z,y) \leq d(x,y) +  \varepsilon\cdot  (d(x,z)\wedge d(z,y)).$$
	 $M$ is said to have property $(Z)$ if
	 each pair of distinct points of $M$ has  property $(Z)$. 
\end{defn}

\begin{defn}
	A complete metric space $M$ is said to be a length space if for every $x,y\in M$ and every $\delta>0$ there exists $z\in M$ such that 
	$$d(x,z) \vee d(z,y) \leq d(x,y)/2 + \delta.$$ 
\end{defn}

Property $(Z)$ was introduced in \cite{Zoriginal} in order to characterize local metric spaces in the compact case. For a complete metric space being local is equivalent to being length \cite[Proposition 2.4]{Zpaper}. 
Papers \cite{Zoriginal} and \cite{Zpaper} are devoted to 
the study of spaces of Lipschitz functions with the Daugavet property. Recall that a Banach space $X$ is said to have the Daugavet property if $\| I+T \| = 1+\| T\|$ for every rank-one operator $T:X \rightarrow X$, where $I:X \rightarrow X$ denotes the identity operator. 

The Daugavet property for Lipschitz-free spaces is characterized as follows: 
\begin{thm}[\cite{Zpaper},\cite{Zoriginal}]
	\label{thmequivalentlegnth}
	Let $M$ be a complete metric space. $\mathcal{F}(M)$ has the Daugavet property if and only if $Lip_0(M)$ has the Daugavet property if and only if $M$ is length.
\end{thm}

It can be easily proved that every length space has property $(Z)$. On the other hand, it was proved in \cite{Zoriginal} that both properties are equivalent in compact metric spaces.

The importance of Property $(Z)$  in the study of Lipschitz-free spaces can be inferred from the following Theorem:

\begin{thm}[\cite{Zpaper}]
\label{thmequivalentZ}	
Let $M$ be a complete metric space. Then the following assertions are equivalent:
\begin{enumerate}
	\item  $M$ has property $(Z)$;
	\item The unit ball of $\mathcal{F}(M)$ does not have strongly exposed points;
	\item The norm of $Lip_0(M)$ does not have any point of G\^{a}teaux differentiability;
	\item The norm of $Lip_0(M)$ does not have any point of Fr\'echet differentiability. 
\end{enumerate}
\end{thm}

For more relations among property $(Z)$, being length and the extremal structure of Lipschitz-free spaces we refer the reader to \cite{AliagaGuirao} and \cite{GLPPR18}.

These results motivated the following question posed in \cite[Question 1]{Zpaper}:

\begin{quest}
	If $M$ is a complete metric space with property $(Z)$, is $M$ length?
\end{quest}

In this paper we provide an affirmative answer to this question:

\begin{THM}
A complete metric space $M$ is length if and only if it has property $(Z)$.
\end{THM}

As a consequence, all conditions appearing in Theorems  \ref{thmequivalentlegnth} and \ref{thmequivalentZ} are equivalent. Moreover, \cite[Proposition 4.9]{Zpaper} asserts that if $M$ is a complete metric length space, then the unit ball of $\mathcal{F}(M)$ does not have preserved extreme points. Notice that every strongly exposed point is a preserved extreme point and that in Lipschitz-free spaces there might be preserved extreme points which are not strongly exposed \cite[Example 6.4]{GLPPR18}. Nevertheless, it follows from the aforementioned results and our Main Theorem that, for Lipschitz-free spaces, the absence of preserved extreme points in the unit ball is equivalent to the absence of strongly exposed points in the unit ball.

Furthermore, every Banach space with the Daugavet property has the strong diameter 2 property \cite[Theorem 4.4]{ALN} and a simple computation shows that the unit ball of a Banach space with the slice diameter 2 property cannot contain strongly exposed points (see \cite{BGLPRZ18} for definitions). Thus, the Main Theorem also provides a complete answer for the scalar case of \cite[Question 3.3]{BGLPRZ18}. Let us summarize in one theorem all the equivalences obtained:

\begin{thm}
Let $M$ be a complete metric space. Then the following assertions are equivalent:
	\begin{enumerate}
		\item  $M$ has property $(Z)$;
		\item $M$ is length;
		\item The unit ball of $\mathcal{F}(M)$ does not have strongly exposed points;
		\item The unit ball of $\mathcal{F}(M)$ does not have preserved extreme points;
		\item $\mathcal{F}(M)$ has the Daugavet property;
		\item $Lip_0(M)$ has the Daugavet property;
		\item The norm of $Lip_0(M)$ does not have any point of G\^{a}teaux differentiability;
		\item The norm of $Lip_0(M)$ does not have any point of Fr\'echet differentiability;
		\item $\mathcal{F}(M)$ has the slice diameter 2 property;
		\item $\mathcal{F}(M)$ has the diameter 2 property;
		\item $\mathcal{F}(M)$ has the strong diameter 2 property.
	\end{enumerate}
\end{thm}

\section{Preliminaries}

In this section we state transfinite versions of several basic facts about sequences in metric spaces.

\begin{defn}
If $\alpha<\omega_1$ is a limit ordinal and $\{x_\gamma : \gamma<\alpha\}$ is a transfinite sequence of points of a metric space $M$, we say that the sequence is Cauchy if for every $\varepsilon>0$ there exists $\beta<\alpha$ such that $d(x_\gamma,x_\delta)<\varepsilon$ whenever $\beta<\gamma<\delta<\alpha$. Similarly, we say that $\{x_\gamma : \gamma<\alpha\}$ converges to $x$ if for every $\varepsilon>0$ there exists $\beta<\alpha$ such that $d(x_\gamma,x)<\varepsilon$ whenever $\beta<\gamma<\alpha$.
\end{defn}

\begin{lem}
A sequence $\{x_\gamma : \gamma<\alpha\}$ is Cauchy (respectively, convergent to a point $x$) if and only if $\{x_{\alpha_n} : n\in\mathbb{N}\}$ is Cauchy (respectively, convergent to $x$) whenever $\alpha_1<\alpha_2<\cdots$ and $\sup_n\alpha_n = \alpha$.
\end{lem}

As an immediate consequence of the above, in a complete metric space every transfinite Cauchy sequence is convergent.

\begin{lem}\label{summablesequence}
Let $\alpha$ be a limit ordinal and $\{x_\gamma : \gamma<\alpha\}$ be a sequence of points in a complete metric space $M$ such that $\{x_\gamma : \gamma<\beta\}$ converges to $x_\beta$ for every limit ordinal $\beta<\alpha$.  Suppose that $$\sum_{\gamma<\alpha}d(x_\gamma,x_{\gamma+1}) < +\infty.$$
Then $\{x_\gamma : \gamma<\alpha\}$ converges to a point $x\in M$ and  $$d(x_0,x) \leq \sum_{\gamma<\alpha}d(x_\gamma,x_{\gamma+1}).$$
\end{lem}

\begin{proof}
We prove it by induction on $\alpha$. For $\alpha=\omega$ this is a well-known fact. Now fix a limit ordinal $\alpha<\omega_1$. We take a sequence of ordinals $0=\alpha_0 <\alpha_1<\alpha_2<\cdots$ such that $\alpha = \sup_n \alpha_n$, and we want to prove that this sequence is Cauchy and its limit satisfies the above inequality. We can enlarge our sequence by adding, between $\alpha_{n-1}$ and $\alpha_n$, the least ordinal $\beta$ such that the interval $[\beta,\alpha_n]$ is finite (if such $\beta$ is in fact between $\alpha_{n-1}$ and $\alpha_n$). Thus, we can suppose that for each $n\geq 1$, either $\alpha_n$ is a limit ordinal, or there are only finitely many ordinals between $\alpha_{n-1}$ and $\alpha_n$. Using either the inductive hypothesis or the triangle inequality, we get that
$$d(x_{\alpha_{n-1}},x_{\alpha_n}) \leq \sum_{\alpha_{n-1}\leq\gamma<\alpha_n}d(x_\gamma,x_{\gamma+1})$$
Thus, the case $\alpha=\omega$ of the lemma applied to the sequence $\{x_{\alpha_n} : n<\omega\}$ gives the desired result.
\end{proof}

For $t,s\in 2^\alpha$, the lexicographical order $t\prec s$ means that $t(\beta)<s(\beta)$ where $\beta$ is the minimal ordinal for which $t(\beta)\neq s(\beta)$. The lexicographical order of $2^\alpha$ is complete, in the sense that every subset $A\subset 2^\alpha$ has a supremum. Such a supremum $s$ can be defined inductively by declaring $s(\gamma) = 1$ if there is some $t\in A$ such that $t|_\gamma = s|_\gamma$ and $t(\gamma)=1$, while $s(\gamma)=0$ if no such $t$ exists.

\section{The Main Theorem}

In order to prove the Main Theorem we only need to check that $M$ is a length space whenever it is a complete metric space with property $(Z)$.

\begin{proof}[Proof of the Main Theorem.]
 Let $M$ be a complete space with property $(Z)$. We fix $x,y\in M$ and two numbers $0<\delta,\eta<1$. We will find $z\in M$ such that  
 $$d(x,z) \vee d(z,y) \leq \frac{d(x,y)}{2(1-\eta)} + \delta.$$ 
 Since $\delta$ and $\eta$ are arbitrarily close to zero, it follows that $M$ is length.
 
 We will define points $x_t,y_t,z_t\in M$ for every $t\in 2^{<\omega_1}$, that is for every function $t:\alpha\To 2=\{0,1\}$ with $\alpha<\omega_1$. Given $x_t$ and $y_t$, then $z_t$ will be defined as follows:
 
 \begin{itemize}
 	\item If $d(x_t,y_t)<\delta$, then $z_t = x_t$
 	\item If $d(x_t,y_t)\geq \delta$, then by using property $(Z)$, we pick $z_t\in M\setminus\{x_t,y_t\}$ such that
 	\begin{eqnarray} \label{ineZ} d(x_t,z_t) + d(z_t,y_t) &\leq& d(x_t,y_t) +  \varepsilon\cdot  (d(x_t,z_t)\wedge d(z_t,y_t))\end{eqnarray}
 	where $\varepsilon = \frac{1}{2}\wedge \frac{\eta}{2} \wedge \frac{\delta\eta}{4 d(x,y)}$. Note that the above inequality also holds when $d(x_t,y_t)<\delta$.
 \end{itemize}

One observation is that, when passing to $z_t$, distances get always reduced:
\begin{eqnarray}
\label{reduced} d(x_t,z_t)\vee d(z_t,y_t) &\leq& d(x_t,y_t), \text{ and inequality is strict if } d(x_t,y_t)\geq \delta.
\end{eqnarray}
This is a direct consequence of (\ref{ineZ}) and the fact that $\varepsilon<1$. A more crucial fact is that, if the distances to the new point $z_t$ are both large, then these distances have to decrease in a fixed amount:

\begin{eqnarray}
\label{fixedamount} d(x_t,z_t)\wedge d(z_t,y_t) \geq \delta &\Longrightarrow& d(x_t,z_t)\vee d(z_t,y_t) \leq d(x_t,y_t) - \delta/2.
\end{eqnarray}
This is because 
\begin{eqnarray*}
	d(x_t,y_t)-d(x_t,z_t) &\geq& d(y_t,z_t) - \varepsilon \cdot  (d(x_t,z_t)\wedge d(z_t,y_t)) \\
	& \geq & (1-\varepsilon)\cdot  (d(x_t,z_t)\wedge d(z_t,y_t)) \geq (1-\varepsilon)\delta \geq \delta/2
\end{eqnarray*}

and we get a similar inequality if we exchange $d(x_t,z_t)$ and $d(z_t,y_t)$.
 
The definition of $x_t$ and $y_t$ is done by induction on the ordinal $dom(t)<\omega_1$, the domain of $t$. Along the construction, we assume inductively that the following inequalities hold for each $\alpha<\omega_1$:

\begin{equation}
\label{ine0} d(x_t,y_t) \leq d(x_{t|_\gamma},y_{t|_\gamma})  \text{ when }\gamma<\alpha=dom(t).\end{equation}
\begin{equation}
\label{ine3}
\sum_{dom(t)=\alpha} d(x_t,y_t)  \leq  d(x,y) +  \varepsilon\sum_{dom(t)<\alpha}d(x_t,z_t)\wedge d(z_t,y_t).\end{equation}
\begin{equation}
\label{ine27}  \text{For every }t,s\in 2^\alpha \text{ with }t\prec s  \begin{cases} d(x_t,y_s) &\leq \displaystyle\sum_{r\in 2^\alpha, t\preceq r\preceq s} d(x_r,y_r),\\[20pt]
d(x_t,x_s) &\leq \displaystyle\sum_{r\in 2^\alpha, t\preceq r\prec s} d(x_r,y_r),\\[20pt]
d(y_t,y_s) &\leq \displaystyle\sum_{r\in 2^\alpha, t\prec r\preceq s} d(x_r,y_r),\\[20pt]
d(y_t,x_s) &\leq \displaystyle\sum_{r\in 2^\alpha, t\prec r\prec s} d(x_r,y_r).\end{cases}
\end{equation}
\begin{equation}
\label{ine2} \varepsilon\sum_{dom(t)<\alpha}d(x_t,z_t)\wedge d(z_t,y_t)
\leq  \eta \sum_{dom(t)=\alpha} d(x_t,y_t).\end{equation}  
 \begin{equation}   
 \label{ine1} \sum_{dom(t)=\alpha} d(x_t,y_t)  \leq  \frac{d(x,y)}{1-\eta}.\end{equation} 
Notice that it is implicit in these inductive assumptions that all sums displayed have finite value, as it follows from (\ref{ine1}) and (\ref{ine2}). Let us proceed to the inductive construction. Set $x_\emptyset=x$ and $y_\emptyset=y$. First, consider $\alpha = \beta+1$ a successor ordinal, and $t$ with $dom(t) = \alpha$. If $t(\beta) = 0$, define $x_t = x_{t|_\beta}$ and $y_t = z_{t|_\beta}$. If $t(\beta) = 1$, define $x_t = z_{t|_\beta}$ and $y_t = y_{t|_\beta}$. If $dom(t)=\alpha<\omega_1$ is a limit ordinal, then we claim that the transfinite sequences $\{x_{t|_\gamma} : \gamma<\alpha\}$ and $\{y_{t|_\gamma} : \gamma<\alpha\}$ will be convergent in $M$, and we will define $x_t$ and $y_t$ as their respective limits. We must prove the convergence of these sequences, but before that let us assume inductively that this was the case for all limit ordinals $\beta<\alpha$ and let us make some observations that help us understand how the construction works. 

\begin{rem}\label{remarkcollapse} If $dom(t)=\beta$ and $d(x_t,y_t)< \delta$, because of the definition of $z_t$ in this case, the construction stabilizes after $t$ in the following sense: For every $s$ with $dom(s)=\gamma>\beta$ and $s|_\beta = t$, then $x_s=y_s=x_t$ unless $s(\gamma)=1$ for all $\gamma\geq \beta$, in which case $x_s = x_t$ and $y_s=y_t$. That is, if we add only ones after $t$ then we keep the same pair $x_t,y_t$, and if we add any 0 we get the trivial pair $(x_t,x_t)$.
\end{rem}

\begin{rem}\label{remarklargepartition}
Given a node $s$ with $dom(s)=\beta$	
$$ \left|\set{\gamma<\beta : d(x_{s|_\gamma},z_{s|_\gamma})\wedge d(z_{s|\gamma},y_{s|_\gamma})\geq \delta}\right| \leq \frac{2 d(x,y)}{\delta}$$
\end{rem}

This means that it cannot happen too often that a pair $(x_t,y_t)$ at large distance splits into two pairs at large distance. The proof follows directly from (\ref{fixedamount}) and (\ref{ine0}), since we have 
$$ 0 \leq d(x_s,y_s) \leq d(x,y)- \frac{\delta}{2} \left|\set{\gamma<\beta : d(x_{s|_\gamma},z_{s|_\gamma})\wedge d(z_{s|\gamma},y_{s|_\gamma})\geq \delta}\right|.$$

\bigskip

Now we prove that  $\{x_{t|_\gamma} : \gamma<\alpha\}$ is a transfinite convergent sequence. The case of $\{y_{t|_\gamma} : \gamma<\alpha\}$ is analogous. We apply Lemma~\ref{summablesequence}. The fact that $\{x_{t|_\gamma} : \gamma<\beta\}$ converges to $x_{t|_\beta}$ when $\beta<\alpha$ is the inductive hypothesis on this definition. If $d(x_{t|_\gamma},y_{t|_\gamma})<\delta$ for some $\gamma<\alpha$, then the sequence $\{x_{t|_\gamma} : \gamma<\alpha\}$ becomes eventually constant and we would be done. So we suppose that $d(x_{t|_\gamma},y_{t|_\gamma})\geq \delta$ for all $\gamma<\alpha$.
It follows from Remark~\ref{remarklargepartition} and from (\ref{reduced}) that there are only finitely many $\gamma<\alpha$ such that $d(x_{t|_{\gamma}},x_{t|_{\gamma+1}})\geq \delta$ (if $d(x_{t|_{\gamma}},x_{t|_{\gamma+1}})\geq \delta$ then $x_{t|_{\gamma+1}}=z_{t|_\gamma}$ and $d(x_{t|_{\beta}},x_{t|_{\beta+1}}) \leq d(z_{t|_{\gamma}},y_{t|_{\gamma}})$ for every $\beta > \gamma$). Thus, to check the hypothesis of Lemma~\ref{summablesequence} it is enough to check that \begin{equation}\label{cauchyseries} \sum_{\gamma<\alpha} \{ d(x_{t|_{\gamma}},x_{t|_{\gamma+1}}) :  0<d(x_{t|_{\gamma}},x_{t|_{\gamma+1}})<\delta\} < +\infty.\end{equation}
Take $\gamma_1<\gamma_2<\cdots<\gamma_m<\alpha$ indices in the above series. Notice that since $ 0<d(x_{t|_{\gamma_k}},x_{t|_{\gamma_k+1}})$ we must have $t(\gamma_k) = 1$. Define $t_k\colon\gamma_m+1\To 2$ by $t_k(\beta) = t(\beta)$ for $\beta<\gamma_k$, $t_k(\gamma_k)=0$ and $t_k(\beta)=1$ if $\beta\geq \gamma_k$. Since $d(x_{t|_{\gamma_k}},x_{t|_{\gamma_k+1}})<\delta$, by Remark~\ref{remarkcollapse}, we will have $x_{t_k} = x_{t|_{\gamma_k}}$, $y_{t_k} = x_{t|_{\gamma_k+1}}$. Thus, using the inductive hypothesis (\ref{ine1}),
$$\sum_{k=1}^m d(x_{t|_{\gamma_k}},x_{t|_{\gamma_k+1}}) = \sum_{k=1}^m d(x_{t_k},y_{t_k}) \leq \sum_{dom(t)=\gamma_m+1} d(x_{t},y_{t}) \leq \frac{d(x,y)}{1-\eta}.$$
Since all finite sums of the series (\ref{cauchyseries}) are bounded by a fixed finite number, the infinite series is finite. This finishes the construction in the limit step.

We still have to check all the inductive assumptions.

Proof of (\ref{ine0}): This is a direct consequence of (\ref{reduced}) and the inductive definitions of $x_t$ and $y_t$. 

Proof of (\ref{ine3}) and proof that all sums appearing in the inductive hypotheses are finite: In the successor case $\alpha = \beta+1$,

\begin{eqnarray*}
\sum_{dom(t)=\alpha} d(x_t,y_t) & = & \sum_{dom(t)=\beta} d(x_t,z_t) + d(z_t,y_t)\\ &\leq& \sum_{dom(t)=\beta} d(x_t,y_t) +  \varepsilon\cdot  (d(x_t,z_t)\wedge d(z_t,y_t))\\
&\leq& d(x,y) + \varepsilon\sum_{dom(t)<\beta}d(x_t,z_t)\wedge d(z_t,y_t) + \varepsilon\sum_{dom(t)=\beta}d(x_t,z_t)\wedge d(z_t,y_t).
\end{eqnarray*}

Notice that the last sum is in fact finite by inductive hypothesis and (for the last summand) inequality (\ref{reduced}). Now we prove (\ref{ine3}) when $\alpha$ is a limit ordinal. First we notice that, by the inductive hypothesis combining (\ref{ine1}) and (\ref{ine2}),
$$\sum_{dom(t)<\alpha}d(x_t,z_t)\wedge d(z_t,y_t) = \sup_{\beta<\alpha} \sum_{dom(t)<\beta}d(x_t,z_t)\wedge d(z_t,y_t) \leq \frac{\eta d(x,y)}{\varepsilon(1-\eta)}<+\infty$$

It is enough to fix $t_1,\ldots,t_m$ with $dom(t)=\alpha$ and $\xi>0$, and check that
$$\sum_{k=1}^m d(x_{t_k},y_{t_k})  \leq  d(x,y) +  \varepsilon\sum_{dom(t)<\alpha}d(x_t,z_t)\wedge d(z_t,y_t) + \xi.$$
Since in the limit case we construct our points as limits, we can find $\beta<\alpha$ such that $d(x_{t_k},x_{t_k|_\beta})+d(y_{t_k},y_{t_k|_\beta})<\xi/m$ for all $k$ and $t_k|_\beta \neq t_r|_\beta$ whenever $k\neq r$. Then, 
\begin{eqnarray*}
	\sum_{k=1}^m d(x_{t_k},y_{t_k})  &\leq& 
\sum_{k=1}^m d(x_{t_k|_\beta},y_{t_k|_\beta}) + \xi \leq \sum_{dom(t)=\beta} d(x_{t},y_{t}) + \xi \\ &\leq& d(x,y) +  \varepsilon\sum_{dom(t)<\beta}d(x_t,z_t)\wedge d(z_t,y_t) + \xi\\ &\leq& d(x,y) +  \varepsilon\sum_{dom(t)<\alpha}d(x_t,z_t)\wedge d(z_t,y_t) + \xi.
\end{eqnarray*}

Proof of (\ref{ine27}): Notice that all four inequalities in (\ref{ine27}) follow from the last one together with
\begin{eqnarray*}
	d(x_t,y_s) &\leq& d(x_t,y_t) + d(y_t,x_s) + d(x_s,y_s),\\ d(x_t,x_s) &\leq& d(x_t,y_t)+d(y_t,x_s),\\ d(y_t,y_s) &\leq& d(y_t,x_s)+d(x_s,y_s).
\end{eqnarray*}

So what we have to prove is that the last inequality holds for $\alpha$ assuming that all inequalities hold for ordinals less than $\alpha$. First we consider the successor case $\alpha = \beta+1$. For $r\in 2^\beta$ and $i\in \{0,1\}$ let $r^\frown i\in 2^\alpha$ be the element that coincides with $r$ below $\beta$ while $r(\beta) = i$. Given $t\prec s$ in $2^\alpha$, they must be of the form $t= u^\frown i$ and $s=v^\frown j$ for some $u,v\in 2^\beta$ with $u\preceq v$. If $u=v$, then $t=u^\frown 0$, $s=u^\frown 1$ and $d(y_t,x_s) = d(z_u,z_u)=0$, so we are done. We suppose that $u\prec v$, and applying the inductive hypothesis,
\begin{eqnarray*}
d(y_{u},x_{v}) &\leq& \sum_{r\in 2^\beta, u \prec r \prec v} d(x_r,y_r)\\
 &=& \sum_{r\in 2^\beta, u \prec r \prec v} d(x_{r^\frown 0},y_{r^\frown 1})\\ \text{(since }y_{r^\frown 0} = z_r = x_{r^\frown 1}\text{)} 
 &\leq & \sum_{r\in 2^\beta, u \prec r \prec v} d(x_{r^\frown 0},y_{r^\frown 0}) + d(x_{r^\frown 1},y_{r^\frown 1})\\ 
 &=&  \sum_{r\in 2^\alpha, u^\frown 1 \prec r \prec v^\frown 0} d(x_r,y_r).
\end{eqnarray*} 

Therefore,
\begin{eqnarray*}d(y_t,x_s) &\leq& d(y_t,y_u) + d(y_u,x_v) + d(x_v,x_s)\\ &\leq& d(y_t,y_u) + \sum_{r\in 2^\alpha, u^\frown 1 \prec r \prec v^\frown 0} d(x_r,y_r) + d(x_v,x_s).\end{eqnarray*}
The desired inequality now follows from a distinction of cases. If $t=u^\frown 0$ then $y_t=x_{u^\frown 1}$ and $y_u = y_{u^\frown 1}$, if $t=u^\frown 1$ then $y_t=y_u$, if $s=v^\frown 0$ then $x_v=x_s$ and if $s=v^\frown 1$ then $x_v = x_{v^\frown 0}$ and $x_s = y_{v^\frown 0}$.

 Now we consider the case when $\alpha$ is a limit ordinal. Again, we fix $t\prec s$ in $2^\alpha$ and we want to prove the last inequality in (\ref{ine27}). For each $\beta\leq \alpha$ we consider
\begin{equation}\label{definitionofBbeta}
	B_\beta = \set{r\in 2^\beta :  d(x_r,y_r)\geq \delta}.\\ 
\end{equation}
Since we already proved that $\sum_{r\in 2^\beta}d(x_r,y_r)<+\infty$ for all $\beta$, the sets $B_\beta$ are all finite. Since we already proved that $\sum_{dom(r)<\alpha}d(x_r,z_r)\wedge d(z_r,y_r)<+\infty$, there are only finitely many $r$ with $dom(r)<\alpha$ and $d(x_r,z_r)\wedge d(z_r,y_r)\geq \delta$. Therefore we can find $\beta_0<\alpha$ such that $d(x_r,z_r)\wedge d(z_r,y_r)<\delta$ whenever $\beta_0<dom(r)<\alpha$. This implies that $|B_\beta|\leq |B_{\gamma}|$ whenever $\beta_0<\gamma<\beta<\alpha$. Therefore we can find $\beta_1>\beta_0$ below $\alpha$ such that $|B_\beta| = |B_\gamma|$ whenever $\beta_1<\gamma<\beta<\alpha$. In fact, we must have that $B_\gamma = \{r|_\gamma : r\in B_\beta\}$ whenever $\beta_1<\gamma<\beta\leq \alpha$. Find $\beta_2>\beta_1$ such that $r|_{\beta_2} \neq r'|_{\beta_2}$ whenever $r\neq r'$ and $r,r'\in B_\alpha\cup\{t,s\}$. This implies that, $t\prec r \prec s$ if and only if $t|_\beta\prec r|_\beta\prec s|_\beta$ whenever $\beta\geq \beta_2$ and $r\in B_\alpha$. So for $\beta>\beta_2$ we have
\begin{eqnarray*}
C_\beta^+ &=& \{r\in 2^\beta : t|_\beta\prec r \prec s|_\beta \text{ and } d(x_r,y_r)\geq \delta\}\\
 &=& \{r|_\beta : r\in C_\alpha^+\}.
\end{eqnarray*}
We also consider
\begin{eqnarray*}
C_\beta^- &=& \{r\in 2^\beta : t|_\beta\prec r \prec s|_\beta \text{ and } d(x_r,y_r)<\delta\}.
\end{eqnarray*}
We fix a positive a number $c>0$ and take $\beta_3>\beta_2$ below $\alpha$ such that
$$\sum_{r\in C_\alpha^+} d(x_{r|_\beta},y_{r|_\beta}) \leq c + \sum_{r\in C_\alpha^+} d(x_r,y_r)$$
whenever $\beta_3<\beta<\alpha$. For $r\in 2^\beta$, let $r^1\in 2^\alpha$ be the function that coincides with $r$ on $\beta$ and takes value $1$ from $\beta$ on. For every $\beta>\beta_3$ we can apply the inductive hypothesis about (\ref{ine27}) and we get that
\begin{eqnarray*}
d(y_{t|_\beta},x_{s|_\beta}) &\leq & \sum_{r\in 2^\beta, t|_\beta \prec r \prec s|_\beta} d(x_r,y_r)\\ &=& \sum_{r\in C_\beta^+}d(x_r,y_r) +  \sum_{r\in C_\beta^-}d(x_r,y_r)\\
&=&  \sum_{r\in C_\alpha^+}d(x_{r|_\beta},y_{r|_\beta}) + \sum_{r\in C_\beta^-}d(x_{r^1},y_{r^1})\\
&\leq & c + \sum_{r\in C_\alpha^+}d(x_{r},y_{r}) + \sum_{r\in C_\alpha^-}d(x_{r},y_{r}) \\ &=& c + \sum_{r\in 2^\alpha, t\prec r \prec s}d(x_{r},y_{r}).
\end{eqnarray*}

Since this holds for every $c>0$ and every $\beta$ with $\beta_3<\beta<\alpha$, the last inequality of (\ref{ine27}) follows.

Proof of (\ref{ine2}): We divide the nonzero summands of the left-hand side into two parts:
\begin{eqnarray*}
 A^- &=& \{t :\ dom(t)<\alpha, \ 0< d(x_t,z_t)\wedge d(z_t,y_t)<\delta\}\\
 A^+ &=& \{t :\ dom(t)<\alpha, \ d(x_t,z_t)\wedge d(z_t,y_t)\geq \delta\}
\end{eqnarray*}

For every $t\in A^-$, consider $\tilde{t}\in 2^\alpha$ given by $\tilde{t}(\gamma) = t(\gamma)$ for $\gamma\in dom(t)$, $\tilde{t}(\gamma)=1$ for $\gamma> dom(t)$ and
$$\tilde{t}(dom(t)) = \begin{cases} 0 & \text{ if } d(x_t,z_t)\wedge d(z_t,y_t) = d(x_t,z_t)\\  1 & \text{ if } d(x_t,z_t)\wedge d(z_t,y_t) = d(z_t,y_t) \end{cases} $$
By Remark~\ref{remarkcollapse}, the construction will stabilize and we will have that
$$d(x_{\tilde{t}},y_{\tilde{t}}) = d(x_t,z_t)\wedge d(z_t,y_t).$$
Moreover, we notice that $\tilde{t}\neq \tilde{s}$ whenever $t\neq s$.  Otherwise, we would have $t = \tilde{t}|_\gamma$ and $s=\tilde{s}|_\beta$ for some ordinals with say $\beta>\gamma$, and then
$$ d(x_s,y_s) = d(x_{\tilde{t}|_\beta},y_{\tilde{t}|_\beta}) \leq d(x_{\tilde{t}|_{\gamma+1}},y_{\tilde{t}|_{\gamma+1}}) = d(x_t,z_t)\wedge d(z_t,y_t)<\delta.$$
By Remark \ref{remarkcollapse} this implies that $d(x_s,z_s)\wedge d(z_s,y_s)= 0$, a contradiction with $s\in A^-$. We conclude that
$$\sum_{t\in A^-}d(x_t,z_t)\wedge d(z_t,y_t) = \sum_{t\in A^-}d(x_{\tilde{t}},y_{\tilde{t}}) \leq \sum_{dom(t)=\alpha}d(x_t,y_t), $$
and since we took $\varepsilon<\eta/2$ we get that 
\begin{equation}\label{Aminus} \varepsilon \sum_{t\in A^-}d(x_t,z_t)\wedge d(z_t,y_t) \leq \frac{\eta}{2} \sum_{dom(t)=\alpha}d(x_t,y_t). \end{equation}

Now we deal with $A^+$. For $t\in A^+$, let $t^0,t^1\in 2^\alpha$ be given by $t^i(\gamma) = t(\gamma)$ for $\gamma\in dom(t)$, and $t^i(\gamma)=i$ for $\gamma\geq dom(t)$. Notice that $x_t = x_{t^0}$ and $y_t = y_{t^1}$.

\begin{eqnarray*}
\sum_{t\in A^+}d(x_t,z_t)\wedge d(z_t,y_t)  &\leq& \sum_{t\in A^+} d(x_t,y_t)  = \sum_{t\in A^+} d(x_{t^0},y_{t^1})\\
\ \text{ by already proven }(\ref{ine27}) \ &\leq& \sum_{t\in A^+} \sum_{t^0\preceq r \preceq t^1, r\in 2^\alpha} d(x_{r},y_{r}) \\
& =& \sum_{r\in 2^\alpha} d(x_r,y_r)\cdot \left|\set{ t\in A^+ : t^0 \preceq r \preceq t^1}\right|
\\
& =& \sum_{r\in 2^\alpha} d(x_r,y_r)\cdot \left|\set{ t\in A^+ : \exists \beta<\alpha : t = r|_\beta }\right| \\
\text{ by Remark } \ref{remarklargepartition} &\leq& 
\sum_{r\in 2^\alpha} \frac{2 d(x,y)}{\delta}d(x_r,y_r) 
\\
(\text{since we chose }\varepsilon < \frac{\delta\eta}{4d(x,y)})  &\leq& 
\frac{\eta}{2\varepsilon}\sum_{r\in 2^\alpha} d(x_r,y_r).
\end{eqnarray*}

This together with (\ref{Aminus}) finishes the proof of (\ref{ine2}).

Proof of (\ref{ine1}): For this, we just apply first (\ref{ine3}) and then (\ref{ine2}) :
\begin{eqnarray*}\sum_{dom(t)=\alpha} d(x_t,y_t)  &\leq&  d(x,y) +  \varepsilon\sum_{dom(t)<\alpha}d(x_t,z_t)\wedge d(z_t,y_t)\\
 &\leq&  d(x,y) + \eta \sum_{dom(t)=\alpha} d(x_t,y_t),
\end{eqnarray*}
from which (\ref{ine1}) follows.

All the inductive hypotheses have been proven. 

Claim: There exists $\xi<\omega_1$ such that $d(x_t,y_t) < \delta$ for all $t\in 2^\xi$.

Proof of the claim: We consider again the finite sets $B_\beta$ introduced in (\ref{definitionofBbeta}). The reasoning in the paragraph below (\ref{definitionofBbeta}) is valid for $\alpha=\omega_1$, and it tells us that we can find  $\beta_1<\omega_1$ such that $|B_\gamma| = |B_\beta|$ and $B_\gamma =  \{r|_\gamma : r\in B_\beta\}$ whenever $\beta_1<\gamma<\beta<\omega_1$. This means that there are $r_1,\ldots,r_n\in 2^{\omega_1}$ such that $B_\gamma =  \{{r_i}|_\gamma : i=1,\ldots,n\}$ whenever $\beta_1<\gamma<\omega_1$. By (\ref{reduced}), we have that $d(x_{{r_i}|_\gamma},y_{{r_i}|_\gamma}) > d(x_{{r_i}|_\beta},y_{{r_i}|_\beta})$ whenever $\beta_1<\gamma<\beta<\omega_1$. There are no strictly decreasing $\omega_1$-sequences of real numbers, so the conclusion is that $n=0$, so $B_\gamma = \emptyset$ for $\gamma>\beta_1$, and that means exactly that $d(x_t,y_t)<\delta$ for $t\in 2^\gamma$ for $\gamma>\beta_1$. The claim is proved.

Now we fix $\xi$ as in the claim. Let
$$t = \sup_{\prec}\set{s\in 2^\xi : \sum_{r\preceq s}d(x_r,y_r)<\frac{d(x,y)}{2(1-\eta)}+\delta}.$$
The supremum is taken in the lexicographical order of $2^\xi$. We claim that $z=x_t$ is the point that we are looking for. Let $\bar{0},\bar{1}\in 2^\xi$ be the constant functions equal to 0 and 1 respectively. On the one hand, using (\ref{ine27}),

\begin{equation}\label{end2}
 d(x,x_t) = d(x_{\bar{0}},x_t) \leq \sum_{r\prec t}d(x_r,y_r) = \sup_{s\prec t}\sum_{r\preceq s}d(x_r,y_r)  \leq \frac{d(x,y)}{2(1-\eta)} + \delta.
\end{equation}

On the other hand, the fact that $t$ is the supremum also implies that
$$ \sum_{r\preceq s} d(x_r,y_r) \geq \frac{d(x,y)}{2(1-\eta)} + \delta \text{ whenever } s\succ t.$$
By (\ref{ine1}), this implies that 
$$ \sum_{r\succ s} d(x_r,y_r) \leq \frac{d(x,y)}{2(1-\eta)}-\delta \text{ whenever } s\succ t.$$
Since $d(x_s,y_s)\leq\delta$ for any $s$, we conclude that
$$ \sum_{r\succeq s} d(x_r,y_r) \leq \frac{d(x,y)}{2(1-\eta)} \text{ whenever } s\succ t,$$
and therefore
$$\sum_{r\succ t} d(x_r,y_r) = \sup_{s\succ t}\sum_{r\succeq s} d(x_r,y_r) \leq \frac{d(x,y)}{2(1-\eta)}.$$
Combining this with (\ref{ine27}) we get that

\begin{equation}\label{end1}
d(x_t,y) = d(x_t,y_{\bar{1}}) \leq \sum_{r\succeq t} d(x_r,y_r) = d(x_t,y_t) + \sum_{r\succ t} d(x_r,y_r) \leq \delta + \frac{d(x,y)}{2(1-\eta)}.
\end{equation}

The inequalities (\ref{end2}) and (\ref{end1}) show that $z=x_t$ is the point that we were looking for, so the proof is over.

\end{proof}

\section*{Acknowledgment}

We would like to thank Abraham Rueda Zoca and Luis Carlos Garc\'ia Lirola for their useful suggestions while preparing this paper.

\end{document}